\def\timestamp{%
Time-stamp: <main.tex: Tuesday 02-02-2021 at 17:48:50 (cet)>}
\def\stripname Time-stamp: <#1 #2>{#2}
\edef\filedate{\expandafter\stripname\timestamp}

\documentclass[a4paper]
               {amsart}

[1994/12/01]

\usepackage{amsfonts}
\usepackage{amsmath}
\usepackage{geometry}
\usepackage{amssymb}
\usepackage{amsthm}
\usepackage{enumitem}
\usepackage{latexsym}
\usepackage{color}
\usepackage[english]{babel}
\usepackage{indentfirst}

\theoremstyle{definition}
\newtheorem{defin}{Definition}[section]
\theoremstyle{plain}

\newtheorem{prop}[defin]{Proposition}
\newtheorem{lem}[defin]{Lemma}
\newtheorem{teo}[defin]{Theorem}
\newtheorem{cor}[defin]{Corollary}

\theoremstyle{remark}

\newtheorem{question}[defin]{Question}

\DeclareMathOperator{\supp}{supp}
\DeclareMathOperator{\diam}{diam}
\DeclareMathOperator{\ult}{Ult}

\newcommand\dotcup{\mathbin{\dot\cup}}
\newcommand\cee{\mathfrak{c}}

\mathchardef\mathhyphen "002D   
\newcommand\Ulim{\mathcal{U}\mathhyphen\!\lim}
\newcommand\plim[1][p]{#1\mathhyphen\!\lim}

\DeclareMathSymbol\Q0{AMSb}{`Q}
\DeclareMathSymbol\R0{AMSb}{`R}
\DeclareMathSymbol\Ss0{AMSb}{`S}
\DeclareMathSymbol\T0{AMSb}{`T}
\DeclareMathSymbol\Z0{AMSb}{`Z}

\newcommand\calA{\mathcal{A}}
\newcommand\calE{\mathcal{E}}
\newcommand\calF{\mathcal{F}}

\newcommand\calM{\mathcal{M}}
\newcommand\calN{\mathcal{N}}
\newcommand\calS{\mathcal{S}}
\newcommand\calT{\mathcal{T}}
\newcommand\calU{\mathcal{U}}
\newcommand\calV{\mathcal{V}}

\begin{document}
	
\title[Large countably compact free Abelian groups]%
  {Countably compact group topologies on arbitrarily large free Abelian groups}
	
\author[M. K. Bellini]{Matheus K. Bellini\dag}
	
\author[K. P. Hart]{Klaas Pieter Hart}
\address[K. P. Hart]{
Faculty EEMCS\\
TU Delft\\
Postbus 5031\\
2600~GA {} Delft\\
the Netherlands}
\email{k.p.hart@tudelft.nl}
\urladdr{http://fa.its.tudelft.nl/\~{}hart}
	
\author[V. O. Rodrigues]{Vinicius O. Rodrigues\ddag}
	
\author[A. H. Tomita]{Artur H. Tomita\P}
\address[M. K. Bellini, V. O. Rodrigues, A. H. Tomita]{Depto de Matem\'atica, 
         Instituto de Matem\'atica e Estat\'istica, 
         Universidade de S\~ao Paulo, Rua do Mat\~ao, 
         1010 -- CEP 05508-090, 
         S\~ao Paulo, SP - Brazil}
\email{matheusb@ime.usp.br}
\email{vinior@ime.usp.br}
\email{tomita@ime.usp.br}

\thanks{\dag 
        The first author received financial support from FAPESP 2017/15709-6.}
\thanks{\ddag 
        The third author received financial support from FAPESP 2017/15502-2. Corresponding Author.}
\thanks{\P 
        The fourth author received financial support from FAPESP 2016/26216-8. The fourth author thanks the second author for the hospitality during his visit to TU Delft in April 2019.}

\subjclass[2010]{Primary 54H11, 22A05; Secondary 54A35, 54G20.}
	
\date{\filedate}
	
	
\commby{}
	
\keywords{Topological group, countable compactness, selective ultrafilter, 
          free Abelian group, Wallace's problem}
	
\begin{abstract}
We prove that if there are $\cee$ incomparable selective ultrafilters 
then, for every infinite cardinal~$\kappa$ such that $\kappa^\omega=\kappa$, 
there exists a group topology on the free Abelian group of cardinality~$\kappa$
without nontrivial convergent sequences and such that every finite power is 
countably compact. 
In particular, there are arbitrarily large countably compact groups. 
This answers a 1992 question of D. Dikranjan and D. Shakhmatov. 
\end{abstract}
	
\maketitle
	
\section{Introduction}
	
\subsection{Some history}
	
It is well known that a non-trivial free Abelian group does not admit a compact
Hausdorff group topology. 
Tomita \cite{tomita2} showed that it does not admit even a group topology 
whose countable power is countably compact .
	
Tkachenko \cite{tkachenko} showed in 1990 that the free Abelian group 
generated by $\cee$~elements can be endowed with a countably compact 
Hausdorff group topology under~CH.  
Tomita \cite{tomita2}, 
Koszmider, Tomita and Watson~\cite{koszmider&tomita&watson}, and 
Madariaga-Garcia and Tomita \cite{madariaga-garcia&tomita} 
obtained such examples using weaker assumptions. 
Boero, Castro Pereira and Tomita obtained such an example using a single 
selective ultrafilter~\cite{boero&castro-pereira&tomitaoneselective}. 
Using $2^\cee$~selective ultrafilters, the example 
in~\cite{madariaga-garcia&tomita} showed the consistency of a countably 
compact group topology on the free Abelian group of cardinality~$2^\cee$. 
All forcing examples obtained so far had their cardinalities bounded 
by~$2^\cee$.
	
Boero and Tomita \cite{boero&tomita2} showed from the existence 
of $\cee$ selective ultrafilters that there exists a free Abelian group 
of cardinality~$\cee$ whose square is countably compact. 
Tomita \cite{tomita2015} showed that there exists a group topology on the 
free Abelian group of cardinality~$\cee$ that makes all its finite powers 
countably compact. 
	
E. van Douwen showed in~\cite{vandouwen} that the cardinality
of a countably compact group cannot be a strong limit of countable cofinality.
	
Using the result in the abstract, we obtain the following:
	
\begin{teo} 
Assume $\mathrm{GCH}$. 
Then a free Abelian group of infinite cardinality~$\kappa$ can be endowed with 
a countably compact group topology (without non-trivial convergent sequences) 
if and only if $\kappa =\kappa^\omega$.
\end{teo}
	
The result above answers a question of Dikranjan and Shakhmatov that was 
posed in the survey by Comfort, Hoffman and Remus~\cite{comfort&hofmann&remus}.

Because of the way our examples are constructed we can raise their weights
in the same way as in the papers~\cite{tomita4} 
or~\cite{castro-pereira&tomita2010} and obtain the following result -- 
the examples in these references are Boolean but the trick is similar.

\begin{teo} 
It is consistent that there is a proper class of cardinals of countable 
cofinality that can occur as the weight of a countably compact free Abelian 
group.
\end{teo}

\subsection{Basic results, notation and terminology}
	
We recall that a topological space is \emph{countably compact} if, 
and only if, every countable open cover of it has a finite subcover. 
	
\begin{defin}\label{defin_p-limit}
Let $\calU$ be a filter on~$\omega$ and let $(x_n : n \in \omega)$ be 
a sequence in a topological space~$X$. 
We say that $x \in X$ is a \emph{$\calU$-limit point} 
of $(x_n : n \in \omega)$ if, 
for every neighborhood~$U$ of~$x$, 
the set $\{n \in \omega : x_n \in U\}$ belongs to~$\calU$.
		
If $X$ is Hausdorff, every sequence has at most one $\calU$-limit
and we write $x =\Ulim (x_n : n \in \omega)$ in that case.
\end{defin}
	
The set of all free ultrafilters on~$\omega$ is denoted by~$\omega^*$. 
The following proposition is a well known result on ultrafilter limits.
	
\begin{prop}
A topological space  is countably compact if and only if each sequence 
in it has a $\calU$-limit point for some $\calU\in \omega^*$.
\end{prop}
	
The concept of almost disjoint families will be useful in our construction.

\begin{defin}
An almost disjoint family is an infinite family~$\calA$  of infinite
subsets of~$\omega$ such that distinct elements of~$\calA$
have a finite intersection. 
\end{defin}
	
It is well known that there exists an almost disjoint family of size 
continuum (see \cite{kunen}).

\begin{defin}
The unit circle group $\T$ will be the metric 
group $(\R / \Z, \delta)$ where the metric $\delta$ is 
given by 
$$
\delta(x + \Z, y + \Z) = \min\{|x - y + a| : a \in \Z\}
$$
for every $x, y \in \R$.
		
Given an open interval $(a, b)$ of $\R$ with $a<b$, we let $\delta((a, b))=b-a$.
		
An arc of $\T$ is a set of the form $I+\Z=\{a+\Z: a \in I\}$, 
where $I$ is an open interval of~$\R$. 
An~arc is said to be proper if it is distinct from $\T$.
		
If $U$ is a proper arc and $U=\{a+\Z: a \in I\}=\{b+\Z: a \in J\}$, then 
the Euclidean length of~$I$ equals the Euclidean length of~$J$, and we define 
the length of~$U$ as $\delta(U)=\delta(I)$. 
We also let $\delta(\T)=1$.
\end{defin}

Given an arc~$U$ such that $\delta(U) \leq \frac{1}{2}$, it follows 
that $\diam_\delta U=\delta(U)$. 
	
Our free Abelian groups will all be represented as directs sums of copies
of the group of integers~$\Z$; we fix some notation.
The additive group of rationals will also be used, so in the following
definition one should read $\Z$ or~$\Q$ for~$G$.	
	
\begin{defin}
If $f$~is a map from a set~$X$ to a group~$G$ then the \emph{support} of~$f$,
denotes $\supp f$ is defined to be the set $\{x \in X : f(x) \neq 0\}$. 

We define $G^{(X)}=\{f \in G^X : |\supp f| < \omega\}$.
		
If $Y$ is a subset of $X$ then, as an abuse of notation, 
we often write $G^{(Y)}=\{x \in G^{(X)}: \supp x \subseteq Y\}$.
		
Given $x \in X$, we denote by $\chi_x$ the characteristic function of~$\{x\}$,
whose support is~$\{x\}$ and which value~$\chi_x (x)=1$. 

For a sequence $\zeta: \omega\to X$ in~$X$ we define 
$\chi_\zeta: \omega\to G^X$ by $\chi_\zeta(n)=\chi_{\zeta(n)}$.

Finally, for $x \in X$, we let $\vec{x}: \omega\to X$ be
the constant sequence with value~$x$.
\end{defin}
	
Note that then $\chi_{\vec x}$ is also constant, with value~$\chi_x$.
	
\begin{defin}
Let $\calU$ be a filter on $\omega$ and $X$ a set. 
We say that the sequences  $f, g \in X^{\omega}$ 
are \emph{$\calU$-equivalent} and write $f\equiv_\calU g$ 
iff $\{n \in \omega : f(n) = g(n)\} \in \calU$.
\end{defin}

It is easy to verify that $\equiv_\calU$ is an equivalence relation. 
We denote the equivalence class of $f \in X^\omega$ by $[f]_\calU$. 
We also denote the set of all equivalence classes by $X^\omega/ \calU$.
		
If $R$ is a ring and $X$ is an $R$-module, then $X^\omega/\calU$ has 
a natural $R$-module structure given by 
$[f]_\calU+[g]_\calU=[f+g]_\calU$, 
$[-f]_\calU=-[f]_\calU$, 
$r\cdot[f]_\calU=[r\cdot f]_\calU$ and 
the class of the zero function as its zero~element.
		
If $p$ is a free ultrafilter, then the \emph{ultrapower} of the $R$-module~$X$ 
by~$p$ is the $R$-module $X^\omega/p$.
	
For the rest of this paper we will fix a cardinal number $\kappa$ that 
satisfies $\kappa^\omega=\kappa$.

Throughout this article, we will work inside ultrapowers of~$\Q^{(\kappa)}$. 
These ultrapowers contain copies of ultrapowers of~$\Z^{(\kappa)}$, which will 
be useful for the construction. 
So it is useful to define some notation.
	
\begin{defin}
Let $p$ be a free ultrafilter on~$\omega$. 
We define $\ult(\Q,p)$ as the $\Q$-vector space $(\Q^{(\kappa)})^\omega/p$ 
and $\ult(\Z, p)=\{[g]_p: g \in \Z^\omega\}$ with the subgroup structure. 
\end{defin}

Notice that each $[g]_p$ in~$\ult(\Z, p)$ is formally an element 
of $(\Q^{(\kappa)})^\omega/p$, not of $(\Z^{(\kappa)})^\omega/p$.
Nevertheless it is clear that $(\Z^{(\kappa)})^\omega/p$ is isomorphic 
to~$\ult(\Z, \kappa)$ via the obvious isomorphism that carries the equivalence 
class of a sequence $g \in (\Z^{\kappa})^\omega$ in~$(\Z^{(\kappa)})^\omega/p$ 
to its class in $(\Q^{(\kappa)})^\omega/p$.

\section{Selective Ultrafilters}

In this section we review some basic facts about selective ultrafilters, 
the Rudin-Keisler order and some lemmas we will use in the next sections.

\begin{defin}
A selective ultrafilter (on $\omega$), also called Ramsey ultrafilter, 
is a free ultrafilter~$p$ on~$\omega$ with the property that for every 
partition $(A_n: n \in \omega)$ of~$\omega$, either there exists~$n$ such 
that $A_n \in p$ or there exists $B \in p$ such 
that $|B\cap A_n|=1$ for every $n \in \omega$. 
\end{defin}
	
The following proposition is well known. We provide \cite{jech} as a reference.

\begin{prop}
Let $p$ be a free ultrafilter on $\omega$. Then the following are equivalent:
\begin{enumerate}[label=\alph*)]
\item $p$ is a selective ultrafilter,
\item for every $f \in \omega^\omega$, there exists $A \in p$ such 
      that $f$ is either constant or one-to-one on~$A$,
\item for every function $f:[\omega]^2\to 2$ there 
      exists $A \in p$ such that $f$~is constant on~$[A]^2$.
\end{enumerate}
\end{prop}
	
The Rudin-Keisler order is defined as follows:

\begin{defin}
Let $\calU$ be a filter on $\omega$ and $f:\omega\to \omega$. 
We define $f_*(\calU)=\{A \subseteq \omega: f^{-1}[A] \in \calU\}$.
\end{defin}
	
It is easy to verify that $f_*(\calU)$ is a filter; 
if $\calU$ is an ultrafilter then so is $f_*(\calU)$; 
if $f, g:\omega\to \omega$, then $(f\circ g)_*=f_*\circ g_*$; 
and $(\text{id}_\omega)_*$ is the identity over the set of all filters. 
This implies that if $f$ is bijective, then $(f^{-1})_*=(f_*)^{-1}$.

\begin{defin}
Let $\calU$ and $\calV$ be filters. 
We say that $\calU\leq \calV$ 
(or $\calU\leq_{\text{RK}} \calV$, if we need to be clear) 
iff there exists $f \in \omega$ such that $f_*(\calU)=\calV$.
		
The \textit{Rudin-Keisler order} is the set of all free ultrafilters 
over~$\omega$ ordered by~$\leq_{\text{RK}}$. 
We say that two ultrafilters $p$ and~$q$ are equivalent 
iff $p\leq q$ and $q\leq p$.
\end{defin}
	
It is easy to verify that $\leq$ is a preorder and that the equivalence 
defined above is indeed an equivalence relation. 
Moreover, the equivalence class of a fixed ultrafilter is the set of all 
fixed ultrafilters, so the relation restricts to~$\omega^*$ without 
modifying the equivalence classes. 
We refer to \cite{jech} for the following proposition:
	
\begin{prop}The following are true:
\begin{enumerate}
\item If $p$ and $q$ are ultrafilters, then $p\leq q$ and $q\leq p$ is 
      equivalent to the existence of a bijection $f:\omega\to \omega$ 
      such that $f_*(p)=q$.
\item The selective ultrafilters are exactly the minimal elements of the 
      Rudin-Keisler order.
\end{enumerate}
\end{prop}
	
This implies that if $f:\omega\to \omega$ and $p$~is a selective 
ultrafilter, then $f_*(p)$ is either a fixed ultrafilter or a 
selective ultrafilter. 
If $f_*(p)$ is the ultrafilter generated by~$n$, then $f^{-1}[\{n\}] \in p$, 
so, in particular, if $f$~is finite to one and $p$~is selective, 
then $f_*(p)$~is a selective ultrafilter equivalent to~$p$.
	
The existence of selective ultrafilters is independent from ZFC. 
Martin's Axiom for countable orders implies the existence of $2^\cee$~pairwise 
incomparable selective ultrafilters in the Rudin-Keisler order.

The lemma below appears in \cite{tomita3}.
	
\begin{lem}\label{countable.ult}
Let $(p_k: k \in \omega)$ be a family of pairwise incomparable 
selective ultrafilters. 
For each $k$ let $(a_{k, i}: i \in \omega)$ be a strictly increasing 
sequence in~$\omega$ such that $\{a_{k,i}: i \in \omega\} \in p_k$ 
and $i<a_{k, i}$ for all $i \in \omega$. 
Then there exists $\{I_k: k \in \omega \}$ such that:
\begin{enumerate}[label=\alph*)]
\item $\{ a_{k,i}: i \in I_k\} \in p_k$, for each $k \in \omega$.
\item $I_j\cap I_j=\emptyset$ whenever $i, j \in \omega$ and $i\neq j$, and 
\item $\{[i,a_{k,i}]:i\in I_k$ and $k\in\omega\}$ is a 
      pairwise disjoint family.
\end{enumerate}
\end{lem}
	
In the course of the construction we will often use families of ultrafilters
indexed by~$\omega$ and finite sequences of infinite subsets of~$\omega$.
The following definition fixes some convenient notation.
	
\begin{defin}
A \emph{finite tower} in $\omega$ is a finite sequence $(A_0, \dots, A_{k-1})$
of infinite subsets of~$\omega$  such that $A_{t+1}\subseteq A_t$ 
for every~$t<k-1$. 
The set of all finite towers in~$\omega$ is called~$\calT$. 
If $T=(A_0, \dots, A_{k-1})$ then $l(T)=A_{k-1}$, the last term of 
the sequence $T$. 
For the empty sequence we write $l(\emptyset)=\omega$.
\end{defin}
	
\begin{lem}\label{ultrafilterenumeration} 
Assume there are $\cee$ incomparable selective ultrafilters. 
Then there is a family of incomparable selective ultrafilters 
$(p_{T,n}: T \in \calT, n \in \omega)$ such that $l(T) \in p_{T,n}$ 
whenever $T \in \calT$ and $n\in \omega$.
\end{lem}
	
\begin{proof}
Index the $\cee$ incomparable selective ultrafilters faithfully as 
$\{q_{T,n}: T \in \calT, n\in \omega\}$. 
For each $T$, let $f_T:\omega \to l(T)$ be a bijection and 
define $p_{T, n}={f_T}_*(q_{T, n})$. 
Since $f$ is one-to-one, it follows that $p_{T,n}$ is a selective ultrafilter 
equivalent to~$q_{T,n}$. 
The family $(p_{T,n}: T \in \calT, n \in \omega)$ is as required.
\end{proof}
	
\section{Main Ideas}

From now on we fix a family $(p_{T, n}: n \in \omega, T\in \calT)$ 
of selective ultrafilters as provided by Lemma~\ref{ultrafilterenumeration}.

The idea will be to use these ultrafilters to assign $p$-limits to 
enough injective sequences in~$\Z^{(\kappa)}$ to ensure countable compactness
of the resulting topology.
We take some inspiration from~\cite{boero&castro-pereira&tomitaoneselective}
where a large independent family was used such that, up to a permutation
every injective sequence in~${\Z}^{(\cee)}$ was part of this family.
Since this group has cardinality~$\cee$, there were indeed enough permutations 
to accomplish this. 
For an arbitrarily large group, we shall consider large linearly independent 
pieces to make sure every sequence has an accumulation point. 
	
The following definition will be used to construct a witness for linearly 
independence in an ultraproduct that does not depend on the free ultrafilter.

\begin{defin} 
Let $\calF$ be a subset of~$({\Z}^{(\kappa)})^\omega$ 
and $A \in [\omega]^\omega$.
We shall call $\calF$ \emph{linearly independent mod~$A^*$} iff for 
every free ultrafilter~$p$ with $A \in p$ 
the set
$$
( [f]_p: f \in \calF) \dotcup ( [\chi_{\vec{\xi}}]_p : \xi < \kappa)
$$ 
is linearly independent in the $\Q$-vector space~$\ult(\Q, p)$, and 
if $[f]_p\neq [g]_p$ whenever $f$ and~$g$ are distinct elements 
of~$\calF$. 
\end{defin}

Notice that it is implicit in our definition that
$\{ [f]_p: f \in \calF\}$ and 
$\{ [\chi_{\vec{\xi}}]_p : \xi < \kappa\}$
are disjoint.
We will abbreviate ``linearly independent mod $A^*$'' to l.i.~mod~$A^*$.
	
An application of Zorn's Lemma will establish the following lemma. 
	
\begin{lem}\label{partial.li} 
Every set of sequences that is l.i.~mod~$A^*$ can be extended to a maximal 
linearly independent set mod~$A^*$.\qed
\end{lem}
	
It should be clear that $A\subseteq B\subseteq \omega$ and $A$ and~$B$ are 
infinite, then a set that is l.i.~mod~$B^*$ is also l.i.~mod~$A^*$. 
Then by using recursion, this easily implies the following corollary:

\begin{cor}\label{baseexistence}
There exists a family $(\calE_T: T \in \calT)$ such that:
\begin{enumerate}
\item For every $T \in \calT$ the set~$\calE_T$ is maximal
      l.i.~mod~$l(T)^*$, and
\item For every $T \in \calT$, if $n\leq |T|$ then 
      $\calE_{T|n}\subseteq \calE_T$.
\end{enumerate}
\end{cor}

We note explicitly that even though $\calE_T$ is only demanded to be maximal
l.i.~mod~$l(T)^*$ it will, because of item~(2), depend on all of~$T$, not just 
on~$l(T)$. 	
	
\begin{lem} \label{partial.generator} 
Let $g$ be an element of $({\Z}^{(\kappa)})^{\omega}$ and let 
$\calE\subseteq(\Z^{(\kappa)})^\omega$ be maximal l.i.~mod~$B^*$. 
Then there exist an infinite subset $A$ of~$B$, a finite subset~$E$ of~$\calE$,
a finite subset~$D$ of~$\kappa$, and sets $\{r_f : f \in E\}$ 
and $\{s_\nu: \nu \in D\}$ of rational numbers such that 
$$
g|_A=\sum_{f \in E}r_f\cdot f|_A +\sum_{\nu \in D} s_\nu \cdot \chi_{\vec \nu}|_A.
$$
\end{lem}
	
\begin{proof}
If $g \in \calE$ or $g=\chi_{\vec \nu}$ for some $\nu<\kappa$, then we are done.
Otherwise, by the maximality of~$\calE$, there exists a free ultrafilter~$p$ 
with $B \in p$ such that  the set
$$
\{[g]_p\}\cup 
\{ [h]_p :\,h \in \calE \} \cup 
\{[\chi_{\vec \xi}]_p: \xi < \kappa \}
$$ 
is not linearly independent. 
		
This means that we can find finite subsets $E$ and $D$ of~$\calE$ and $\kappa$
respectively and finite sets $\{r_f : f \in E\}$ and 
$\{s_\nu: \nu \in D\}$ of rational numbers  such that 
$$
[g]_p=\sum_{f \in E}r_f\cdot[f]_p+\sum_{\nu \in D} s_\nu\cdot[\chi_{\vec \nu}]_p.
$$ 
Now choose $A\in p$ with $A\subseteq B$ that witnesses this equality.
\end{proof}
	
\begin{cor}
If $\calE\subseteq(\Z^{(\kappa)})^\omega$ is maximal l.i.~mod~$B^*$, 
then $|\calE|=\kappa$.
\end{cor}
\begin{proof}
First notice that $|\calE|\le|(\Z^{(\kappa)})^\omega|=\kappa^\omega=\kappa$.
Assume $|\calE|<\kappa$.
Then the set $C=\bigcup\{\supp f(n): n \in \omega, f \in \calE\}$ has 
cardinality less than~$\kappa$.

Take some injective sequence $\langle \xi_n:n\in\omega\rangle$ 
in~$\kappa\setminus C$ and define $g: \omega\to \Z^{(\kappa)}$ by
$g(n)=\chi_{\xi_n}$ for all~$n$.
Clearly then $\bigcup\{\supp g(n): n \in \omega\}$ is disjoint from~$C$, 
all values of~$g$ are non-zero and the values have disjoint supports.

Apply Lemma \ref{partial.generator} to obtain sets $A$, $E$, $D$, 
$\{r_f:f\in E\}$, and $\{s_\nu:\nu\in D\}$ such that 
$$
g|_A=\sum_{f \in E}r_f\cdot f|_A +\sum_{\nu \in D} s_\nu\cdot \chi_{\vec \nu}|_A.
\eqno(*)
$$
Since $A$ is infinite and $D$ is finite, there is a $k\in A$ such 
that $\xi_k\notin D$. 
Now $f(k)(\xi_k)=0$ when $f\in E$ because $\xi_k\notin C$,
and $\chi_{\vec\nu}(k)(\xi_k)=0$ when $\nu\in D$ because $\xi_k\notin D$, and 
also $g(k)(\xi_k)=1$, which contradicts~$(*)$.
\end{proof}

Henceforth we fix a family $(\calE_T: T \in \calT)$ as in 
Corollary~\ref{baseexistence} and enumerate each $\calE_T$ faithfully 
as $\calE_T=\{f^T_\xi: \kappa \leq \xi < \kappa +\kappa\}$.
	
\begin{defin} 
For each $T \in \calT$ and $n\in \omega$, we denote by $G_{T,n}$ the
intersection of~$\ult(\Z, p_{T, n})$ and the free Abelian group generated 
by 
$\{ \frac{1}{n!}[f^T_\xi]_{{ p}_{T,n}}: \kappa \leq \xi <\kappa +\kappa \} \cup 
  \{\frac{1}{n!}[\chi_{\vec\xi}]_{{p}_{T,n}}: \xi <\kappa \}$.
\end{defin}

For the next lemma, we are going to use the following proposition:

\begin{prop}\label{propcyclicfuchs}
If $G$ is an abelian group and $H$ is a subgroup of $G$ such that $G/H$ is 
an infinite cyclic group, then there exists $a \in G$ such 
that $G=H\oplus\langle a\rangle$.
\end{prop}
	
A proof may be found in \cite[14.4]{fuchs1970infinite}. 
This is not the statement of the theorem but it is exactly what is proved 
by the author.
	
The main idea of the proof of the following lemma is to mimic the well known 
proof of the fact that every subgroup of a free abelian group is free.

\begin{lem} 
The group $G_{T,n}$ has a basis of the form 
$\{[\chi_{\vec\xi}]_{p_{T,n}}:\xi<\kappa\}\dotcup\{[f]_{p_{T,n}}:f\in\calF_{T,n}\}$
for some subset $\calF_{T,n}$ of~$({\Z}^{(\kappa)})^\omega$.  
\end{lem}
	
\begin{proof}
Let $H_\mu$ the the group generated by 
$\{\frac{1}{n!}[\chi_{\vec\xi}]_{p_{T, n}}: \xi <\mu \}$ if $\mu\le\kappa$ 
and by the union of $\{\frac{1}{n!}[\chi_{\vec\xi}]_{p_{T, n}}: \xi <\kappa\}$
and $\{ \frac{1}{n!}[f^T_\xi]_{p_{T, n}}: \kappa \leq \xi <\mu \}$
when $\kappa<\mu\leq\kappa +\kappa$. 

Let $G_\mu = H_\mu \cap\ult(\Z, p_{T, n})$ for all~$\mu$.

For every $\mu<\kappa+\kappa$ we shall find $h_\mu$ so that
$G_{\mu+1}=G_\mu\oplus\bigl<\{[h_\mu]_{p_{T, n}}\}\bigr>$, as follows.

For $\mu<\kappa$ the group $G_\mu$ is generated 
by~$\{[\chi_{\vec \xi}]_{p_{T, n}}: \xi<\mu\}$, 
so $G_{\mu+1}=G_{\mu}\oplus\langle\{[\chi_{\vec \mu}]\}\rangle$ and we have
$h_\mu=\chi_{\vec \mu}$.
		
For $\mu\ge\kappa$ observe that $G_{\mu+1}\cap H_\mu=G_\mu$, so:
$$
\frac{G_{\mu+1}}{G_\mu}=
    \frac{G_{\mu+1}}{{G_{\mu+1}}\cap H_\mu}\approx
    \frac{G_{\mu+1}+H_\mu}{H_\mu}\leq
    \frac{H_{\mu+1}}{H_\mu}.
$$
The group $\frac{H_{\mu+1}}{H_\mu}$ is cyclic infinite, 
so either $\frac{G_{\mu+1}}{G_\mu}$ is infinite and cyclic or $G_{\mu+1}=G_{\mu}$. 
By Proposition \ref{propcyclicfuchs} there exists $a_{\mu} \in G_{\mu+1}$ such 
that $G_{\mu+1}=G_{\mu}\oplus\langle\{a_\mu\}\rangle$ 
(and $a_\mu=0$ in case $G_{\mu+1}=G_\mu$). 
Take $h_\mu$ such that $[h_\mu]_{p_{T, n}}=a_\mu$.
		
For every $\mu<\kappa+\kappa$, it follows that 
$G_{\mu+1}=G_{\mu}\oplus\bigl<\{[h_\mu]_{p_{T, n}}\}\bigr>$. 
Since $G_{T, n}=\bigcup_{\mu<\kappa+\kappa} G_\mu$, it follows 
that $G_{T, n}=\bigoplus_{\mu<\kappa+\kappa}\bigl<\{[h_\mu]_{p_{T, n}}\}\bigr>$. 

The set 
$\calF_{T, n}=\{h_{\mu}: \kappa\leq \mu<\kappa+\kappa, [h_{\mu}]_{p_{T, n}}\neq 0\}$ 
is as required.
\end{proof}
	
For the rest of this article we fix such a set~$\calF_{T, n}$ as above
for each pair $(T, n)$ in~$\calT \times \omega$.

The next lemma makes good on the promise from the beginning of this section
as it shows how to make our topology countably compact.
	
\begin{lem} \label{condition.cc} 
Assume that for every pair $(T,n)$ in $\calT\times\omega$ every sequence~$f$
in~$\calF_{T,n}$ has a $p_{T, n}$-limit in ${\Z}^{(\kappa)}$.
Then every finite power of ${\Z}^{(\kappa)}$ is countably compact.
\end{lem}
	
\begin{proof}
A sequence in some finite power of ${\Z}^{(\kappa)}$ is represented
by finitely members of~$({\Z}^{(\kappa)})^{\omega}$, say $g_0, \dots, g_m$.
We show that there is one ultrafilter~$p$ such that $\plim g_i$ exists
for all~$i$, namely $p_{T,n}$ for a suitable~$T$ and~$n$.
		
Recursively, we define a tower $T=(A_0, \dots, A_m)$ and 
for $i\leq m$ finite subsets $E_i$ and $D_i$ of $\calE_{T|_i}$ and $\kappa$
respectively together with finite sets $(r_f^i: f \in E_i)$ and 
$(s_\nu^i: \nu \in D_i)$ of rational numbers such that
$$
g_i|_{A_i}=\sum_{f \in E_i}r_f^i\cdot f|_{A_i}
        +\sum_{\nu \in D_i} s_\nu^i\cdot \chi_{\vec \nu}|_{A_i}\eqno(*)
$$
For $i=0$, use Lemma \ref{partial.generator} applied to~$\calE_{\emptyset}$
to obtain $A_0$, $E_0$, $D_0$, $(r^0_f : f \in E_0)$ and 
$(s^0_\nu: \nu \in D_0)$ such that $(*)$ holds with~$i=0$.
		
To go from $i$ to~$i+1$ apply Lemma~\ref{partial.generator} 
to~$\calE_{(A_0, \dots, A_i)}$
to obtain $A_{i+1}$, $E_{i+1}$, $D_{i+1}$, $(r^{i+1}_f : f \in E_{i+1})$,
and $(s^{i+1}_\nu: \nu \in D_{i+1})$ so that $(*)$~holds for~$i+1$.

Let $A=A_m$ and let $n$ be sufficiently large so that $n! r_f^i$ 
and $n! s_\nu^i$ are integers, for all $i\leq m$,  $f \in E_i$, 
and $\nu \in D_i$.
Then $g_i|_A=\sum_{f \in E_i}n!\cdot r_f^i\cdot(\frac{1}{n!}\cdot f)|_A 
 +\sum_{\nu \in D_i}n!\cdot s_\nu^i\cdot(\frac{1}{n!}\cdot\chi_{\vec \nu})|_A$
for all~$i$. 

As $l(T)=A\in p_{T, n}$ and for each $E_i$ is a subset of~$\calE_T$, it 
follows that $[g_i]_{p_{T, n}}\in G_{T, n}$. 
Therefore, each $ [g_i]_{p_{T, n}}$ is an integer combination of 
$\{ [f]_{p_{T, n}}: f \in \calF_{T, n} \} \cup\{ [\chi_\xi]_{p_{T, n}}:\xi<\kappa\}$. 
Then, by hypothesis, it follows that each $g_i$ has a $p_{T, n}$-limit. 
This completes the proof.
\end{proof}

\section{Constructing homomorphisms}

Through this section, we let $G={\Z}^{(\kappa)}$ 
and  we let $\{h_\xi : \omega \leq \xi < \kappa\}$ be an enumeration 
of~$G^\omega$ such that $\supp h_\xi(n) \subseteq \xi$ whenever $n\in \omega$ 
and $\omega \leq \xi <\kappa$,
and so that each element of $G^\omega$ appears at least $\cee$ many times.
	
\begin{lem}
There exists a family $(J_{T, n}: T \in \calT, n \in \omega)$ of pairwise 
disjoint subsets of $\kappa$ such that $\{h_\xi: \xi \in J_{T, n}\}=\calF_{T, n}$.
\end{lem}
	
\begin{proof}
For each $f \in G^\omega$ there is an injective map 
$\phi_f:\calT\times \omega\to \{\xi \in \kappa:f=h_\xi\}$. 
Let $J_{T, n}=\{\phi_f(T, n): f \in \calF_{T, n}\}$ and we are done.
\end{proof}
	
For the rest of this section, we fix a family 
$(J_{T, n}: T\in\calT, n \in \omega)$ as above.
	
The following lemma is the key to the main result.
	
\begin{lem} \label{countable.hom}
Assume we have a non-zero element~$d$ of $G$, an injective sequence~$r$ in~$G$,
and a countably infinite subset~$D$ of~$\kappa$ such that 
\begin{enumerate}
\item $\omega\cup\supp d\cup \bigcup_{n \in \omega}\supp r(n) \subseteq D$, 
\item $D\cap J_{T, n}\neq \emptyset$ for infinitely many $(T, n)$'s and, 
\item $\supp h_\xi(n)\subseteq D$ for all $n \in \omega$ and 
      $\xi \in D \setminus \omega$
\end{enumerate}
Then there exists a homomorphism $\phi:\Z^{(D)}\to\T$ such that:
\begin{enumerate}
\item $\phi(d)\neq 0$
\item $\plim[p_{T, n}]_k \phi(h_\xi(k))=\phi(\chi_\xi)$, 
      whenever $T\in\calT$, $n\in \omega$, and $\xi \in D\cap J_{T, n}$.
\item $\phi\circ r$ does not converge.
\end{enumerate}
\end{lem}

Before proving this lemma, we show how to use it to prove the main result. 
First, we use it to prove another lemma:
	
\begin{lem} \label{notcountable.hom}
Assume $d$ is a non-zero element of~$G$ and $r$~is an injective sequence in~$G$.
Then there exists a homomorphism $\phi: \Z^{(\kappa)}\to\T$ such that 
\begin{enumerate}
\item $\phi(d)\neq 0$
\item $\plim[p_{T, n}]_k\phi(h_\xi(k))=\phi(\chi_\xi)$, 
      whenever $T \in \calT$, $n\in \omega$ and $\xi \in J_{T, n}$.
\item $\phi\circ r$ does not converge.
\end{enumerate}
\end{lem}
	
\begin{proof}
Using a closing-off argument construct a countable subset~$D$ of~$\kappa$
that intersects infinitely many sets $J_{T,n}$, and that
contains $\omega$, $\supp d$, $\supp r(n)$ for all~$n$ as well 
as $\supp h_\xi(n)$ whenever $\xi\in D\setminus\omega$ and $n\in\omega$.

By the previous Lemma, there exists a homomorphism 
$\phi_0: \Z^{(D)}\to\T$ such that 
$\phi_0(d)\neq 0$, 
$\phi\circ r$ does not converge, and 
$\plim[p_{T, n}]_k\phi_0(h_\xi(k))=\phi_0(\chi_\xi)$ 
whenever $T\in\calT$, $n\in \omega$ and $\xi \in D\cap J_{T, n}$.
		
We let $\langle\alpha_\delta: \delta<\kappa\rangle$ be the monotone
enumeration of~$\kappa\setminus D$. 
For $\gamma\leq \kappa$, 
let $D_\gamma=D\cup\{\alpha_\delta: \delta<\gamma\}$. 
So $D_0=D$ and $D_\kappa=\kappa$.
		
Recursively, we construct, for $\gamma\leq \kappa$, an increasing sequence
of homomorphisms $\phi_\gamma: \Z^{(D_\gamma)}\to \T$ such that 
$\plim[p_{T, n}]_k\phi_\gamma(h_\xi(k))=\phi_\gamma(\chi_\xi)$ whenever
$T \in \calT$, $n\in \omega$ and $\xi \in D_\gamma\cap J_{T, n}$. 
Our homomorphism~$\phi$ will be~$\phi_\kappa$. 
The basis step~$0$ is already done, and for limit steps, we just unite all 
previous homomorphisms. 

To define $\phi_{\gamma+1}$ given~$\phi_\gamma$ it suffices to specify
the value $\phi_{\gamma+1}(\chi_{\alpha_\gamma})$.
		
If $\alpha_\gamma \in J_{T, n}$ for some $T \in \calT$ and $n \in \omega$
then we put 
$\phi_{\gamma+1}(\chi_{\alpha_\gamma})=\plim[p_{T, n}]_n \phi_\gamma(h_\gamma(n))$.
This is well defined because 
$\supp h_\gamma(n)\subseteq \gamma \subseteq D_\gamma$ for all~$n$ and
because $\T$~is compact.  
In the other case let $\phi_{\gamma+1}(\chi_{\alpha_\gamma})=0$.
\end{proof}

We can now prove our main result.	
	
\begin{teo} 
Assume the existence of pairwise incompatible $\cee$ selective ultrafilters 
and that $\kappa$ is an infinite cardinal such that~$\kappa^\omega$. 
Then the free abelian group of cardinality~$\kappa$ has a Hausdorff group 
topology without nontrivial converging sequences such that all of its 
finite powers are countably compact.
\end{teo}
	
\begin{proof} 
Following the notation of the rest of the article, 
given $d \in G\setminus \{0\}$ and an injective sequence~$r$ in~$G$,
Lemma~\ref{notcountable.hom} provides a homomorphism 
$\phi_{d, r}: G\to \T$ such that $\phi_d(d)\neq 0$,
such that $\phi_{d, r}\circ r$ does not converge, and such that
$\plim[p_{T,n}]_k\phi_{d,r}(h_{\xi}(k))=\phi_{d,r}(\chi_\xi)$ whenever
$T\in\calT$, $n\in\omega$ and $\xi\in J_{T,n}$.
We give $G$ the initial topology generated by the collection of homomorphisms 
$\{\phi_{d, r}: d \in G\setminus \{0\}$, $r\in G^\omega$~is injective$\}$
thus obtained and the natural topology of~$\T$.
	
Since the initial topology generated by any collection of group homomorphisms 
is a group topology we do indeed obtain a group topology. 
Since $\T$~is Hausdorff and for every $d\neq 0$ there are many $\phi_{d, r}$ 
with $\phi_{d,r}(d)\neq 0$ it follows at once that our topology is Hausdorff.
		
To see that every finite power of~$G$ is countably compact we 
use Lemma~\ref{condition.cc}. 

Given $T \in \calT$, $n \in \omega$ and $f \in \calF_{T, n}$, 
there exist $\xi \in J_{T, n}$ such that $h_{\xi}=f$. 
For every $d \in G\setminus \{0\}$ and injective $r \in G^\omega$, 
we have $\plim[p_{T, n}]_n \phi_{d, r}(h_\xi(n))=\phi_{d, r}(\chi_\xi)$. 
So $\plim[p_{T,n}]f(n)=\chi_\xi$ and we are done.
		
Since for a given injective sequence $r$ and any $d \in G^\omega$ the
sequence $\phi_{d, r}\circ r$ does not converge and $\phi_{d, r}$ is continuous, 
it follows that $r$ does not converge. 
So $G$ has no nontrivial convergent sequences.
\end{proof}
	
Towards the proof of Lemma~\ref{countable.hom} we formulate a definition 
and a (very) technical lemma.
	
\begin{defin} 
Let $\epsilon >0$.
An \emph{$\epsilon$-arc function} is a function~$\psi$  
from~$\kappa$ into the set of open arcs of~$\T$ (including $\T$ itself) 
such that for all~$\alpha$ either $\psi(\alpha)=\T$ or the length
of~$\psi(\alpha)$ is equal to~$\epsilon$, and the set
$\{\alpha \in \kappa: \psi(\alpha)\neq \T\}$ is finite. 
We will call this finite set the support of~$\psi$ and denote it 
by $\supp \psi$.
	
Given two arc functions $\psi$ and $\varrho$ we write $\psi\leq \varrho$ 
if $\overline{ \psi(\alpha)} \subseteq \varrho(\alpha)$  
or $\psi(\alpha) = \varrho(\alpha)$ for each $\alpha \in \kappa$.
\end{defin}
	
We shall obtain our homomorphisms using limits of such arc functions.
The following lemmas are instrumental in its construction.
	
The following result follows from an argument implicit in the construction 
of~\cite{boero&castro-pereira&tomitaoneselective}, but it may be difficult 
to extract it from that paper. 
We postpone its rather technical proof to the next section.
	
\begin{lem} \label{step} 
Let $p$ be a selective ultrafilter and 
$\calF$ a finite subset of~$G^\omega $ such that the set 
$\{[f]_p: f \in \calF \} \cup \{ [\chi_{\vec\alpha} ]_p: \alpha < \kappa \}$ 
is linearly independent. 

Then for a given $\epsilon>0$ and a finite subset~$E$ of $\kappa$ 
there exist $A\in p$ and a sequence $(\delta_n: n \in A)$ of positive 
real numbers such that
\begin{itemize}		
\item[$(\star)$] whenever $\{U_f: f \in \calF\}$ is a family of arcs of 
     length~$\epsilon$ and $\varrho$~is an arc function of length at 
     least~$\epsilon$ with $\supp \varrho \subseteq E$ there exist 
     for each $n \in A$ a $\delta_n$-arc function $\psi_n \leq \varrho$ such 
     that $\supp \psi_n =\bigcup_{f \in \calF}\supp f(n) \cup E$, and 
     $\sum_{\mu \in \supp f }f(n)(\mu)\cdot\psi_n(\mu)\subseteq U_f$ 
     for each $f \in \calF$.
\end{itemize}
\end{lem}

Now we proceed to prove Lemma~\ref{countable.hom}. 
We will use the following lemma:
	
\begin{lem} \label{countable.hom.aux}
Let $(\calF^k:k\in\omega)$ be a sequence of countable subsets of~$G^\omega$ 
and let $(p_k:k\in\omega)$ is a sequence of pairwise incomparable selective 
ultrafilters such that for each $k\in\omega$ the set
$\{[f]_{p_k}:f\in\calF^k\}\dotcup\{[\chi_{\vec\xi}]_{p_k}:\xi\in\kappa\}$
is linearly independent and $[f]_{p_k}\ne[g]_{p_k}$ whenever $f\ne g$ 
in~$\calF^k$.
Furthermore let for every $f\in\bigcup_k\calF^k$ an ordinal~$\xi_f$ in~$\kappa$
be given.
In addition let $d$ and $d'$ be non-zero in~$G$ and with disjoint supports.
Finally, let $D$ be a countable subset of~$\kappa$ that contains
$\omega\cup\supp d\cup\supp d'$ and $\bigcup_n\supp f(n)$ for 
every $f\in\bigcup_k\calF^k$.

Then there exists a homomorphism $\phi:\Z^{(D)}\to\T$ such that 
$\phi(d)\neq0$, $\phi(d')\neq0$ and 
$\plim[p_k]_n\phi(f(n))=\phi(\chi_{\xi_f})$, whenever $k\in\omega$ 
and~$f\in\calF^k$.
\end{lem}

\begin{proof}
Write $D$ as the union of an increasing sequence $(D_n:n\in\omega)$ of finite
nonempty subsets, and take a similar sequence $(\calF^k_n:n\in\omega)$ for 
each~$\calF^k$.
	
Take a sufficiently small positive number $\epsilon_0$ and an 
$\epsilon_0$-arc function~$\varrho_*$ such that 
$\supp d\cup\supp d'\subseteq \supp\varrho_*$ 
and $0\notin\overline{\sum_{\mu\in \supp d}d(\mu)\varrho_*(\mu)}
        \cup\overline{\sum_{\mu\in \supp d'}d'(\mu)\varrho_*(\mu)}$.

Let $E_0=\supp \varrho_* \cup D_0$ and $B^k_0=\omega$ for each $k \in \omega$.

We will define, by recursion, for $m \in \omega$: 
finite sequences $(B^k_m: 0\leq k \leq m)$, 
finite sets $E_m \subseteq \kappa$, and
real numbers $\epsilon_m>0$ satisfying:
\begin{enumerate}
\item For all $k$ and $m$ in $\omega$ we have $B^k_m \in p_k$,
\item For each $m\geq 1$ and $k\leq m$, 
      we have a sequence~$(\delta^k_{m, n}: n \in \omega)$  of positive real 
      numbers such that: 
      if $(U_f: f \in \calF^k_m)$~is a family of arcs of length~$\epsilon_{m-1}$
      and $\varrho$~is an arc function of length~$\epsilon_{m-1}$ and 
      $\supp \varrho \subseteq E_{m-1}$ then for each $n\in\omega$ there 
      exists a $\delta^k_{m, n}$-arc function~$\psi$ with $\psi\le\varrho$,
      and $\supp \psi=\bigcup_{f \in \calF^k_m}\supp f(n)\cup E_{m-1}$, 
      and $\sum_{\mu \in \supp f}f(n)(\mu)\psi(\mu)\subseteq U_f$ for each 
      $f \in \calF^k_m$. 
\item For all $k$ and $m$ we have $B^k_{m+1}\subseteq B^k_m$.
\item $\epsilon_{m+1}= \frac{1}{2}\min(\{\delta^k_{l,n}: k\leq l\leq m+1$ and 
      $ n \in (m+2) \cap B^k_l\}\cup \{\epsilon_m\})$.
\end{enumerate}
Suppose we have defined $B^k_l$ for all~$k$ as well as $E_l$ and $\epsilon_l$
for all $l\le m$.
As will be clear from the step below the set $B^k_m$ is only non-trivial
whenever $k\le m$.
Therefore we let $B^k_{m+1}=B^k_m=\omega$ for $k>m+1$ and we concentrate
on the case~$k\le m+1$.

Let $k\leq m+1$. 
By~Lemma \ref{step}, there exist $B^k_{m+1}\in p_k$ and 
$(\delta^k_{m+1,n}: n \in \omega)$ that satisfy~(2) for~$m+1$.
Without loss of generality we can assume that $B^k_{m+1}\subseteq B^k_m$.
		
Condition~(4) now specifies $\epsilon_{m+1}$.
	
Setting 
$E_{m+1}= E_m \cup 
        \bigcup\{\supp f(k): k\leq m, f \in \bigcup_{k\leq m+1} \calF_{m+1}^k \} 
        \cup D_{m+1}$
completes the recursion.

\smallskip		
For each $k \in \omega$, apply the selectivity of $p_k$, to choose
an increasing sequence $(a_{k, i}: i \in \omega)$ with 
$\{a_{k,i}: i \in \omega\} \in p_k$ and such 
that $a_{k,i} \in B^k_i$ and $a_{k, i}>i$ for all~$i$.
		
Next apply Lemma~\ref{countable.ult} and let $(I_k: k \in \omega)$ be a 
sequence of pairwise disjoint subsets of~$\omega$ such that 		
$\{a_{k,i}: i \in I_k\}\in p_k$ and the family of intervals 
$\{[i,a_{k,i}] : k\in \omega, i \in I_k\}$ is pairwise disjoint. 
Without loss of generality we can assume that $k< \min I_k$.
		
Enumerate $\bigcup_{k \in \omega}I_k$ in increasing order as~$(i_t:t\in\omega)$. 
For each $t \in \omega$, let $k_t$ be such that $i_t \in I_{k_t}$. 
Thus, for each~$t$ we have $i_t \in I_{k_t}$, 
and hence $i_t\geq \min I_{k_t}>k_t$ and $a_{k_t, i_t}>i_t$.
	
By recursion we define a sequence of arc functions,
$(\varrho_{i_t}: t \in \omega)$,
such that $\varrho_{i_0}\leq \varrho_*$ and $\varrho_{i_{t+1}}\leq \varrho_{i_t}$.
		
We start with $t=0$.  
Then we have $k_0<i_0<a_{k_0,i_0}$, and $a_{k_0,i_0}\in B^{k_0}_{i_0}$,
and $\epsilon_{i_0-1}\leq \epsilon_0$.  

Since $\varrho_*$ has length at least $\epsilon_{i_0-1}$,  there exists an arc 
function~$\varrho_{i_0}$ of length~$\delta^{k_0}_{i_0,a_{k_0,i_0}}$ such that 
$\sum_{\mu\in\supp f}f(a_{k_0,i_0})(\mu)\varrho_{i_0}(\mu)\subseteq\varrho_*(\xi_f)$,
for each~$f \in \calF_{i_0}^{k_0}$.  
We have by the definition that $\delta^{k_0}_{i_0,a_{k_0,i_0}}> \epsilon_{i_1-1}$.

Suppose $t> 0$ and that $\varrho_{i_{t-1}} $ has been defined with length at 
least~$\epsilon_{i_{t-1}}$.
		
Apply item (2) to 
the arc function~$\varrho_{i_{t-1}}$, 
the finite set~$\calF=\calF ^{k_t}_{i_t}$,  
the number~$\epsilon_{i_{t-1}}$, 
the finite set~$E_{i_{t-1}}$, 
the arcs $U_f=\varrho_{i_{t-1}}(\xi_f)$ for $f \in\calF_{i_t}^{k_t}$, 
and $n=a_{k_t,i_t}\in B^{k_t}_{i_t}$
to obtain an arc function $\varrho_{i_t}\leq \varrho_{i_{t-1}}$ such that
$\sum_{\mu \in \supp f}f(a_{k_t,i_t})(\mu) \varrho_{i_t}(\mu)\subseteq
    \varrho_{i_{t-1}}(\xi_f)$ for all $f \in \calF_{i_t}^{k_t}$,
and $\varrho_{i_t}$ has length $\delta^{k_t}_{i_t,a_{k_t,i_t}}$. 

Because $k_t<i_t<a_{k_t,i_t}\leq i_{t+1}-1$ and $a_{k_t,i_t} \in B^{k_t}_{i_t}$
we get $\delta^{k_t}_{i_t,a_{k_t,i_t}}>\epsilon_{i_{t+1}-1}$.
		
If $\xi \in D_{i_t}$ then $\xi \in \supp \varrho_{i_t}$ and the length 
of $\varrho_{i_t}(\xi)$ is not greater than~$\epsilon_{i_t-1}$
which in turn is not larger than $\frac{1}{2^{i_{t-1}}}\leq \frac{1}{2^t}$.
		
It follows that for all $\xi\in D$ the intersection 
$\bigcap_{t\in\omega}\varrho_{i_t}(\xi)$
consists of a unique element; 
we define $\phi(\chi_\xi)$ to be that element
and extend $\phi$ to a group homomorphism.
		
By construction 
$\phi(f(a_{k_t,{i_t}}))$ is 
in $\sum_{\mu \in \supp f}f(a_{k_t,i_t})(\mu) \varrho_{i_t}(\mu)$
which is a subset of~$\varrho_{i_{t-1}} (\xi_f)$
whenever $f \in \calF_{i_t}^{k_t}$. 
Therefore, the sequence	$(\phi(f(a_{k,i})))_{i\in I_k}$ converges to 
$\phi(\chi_{\xi_f})$, for each $k\in \omega$ and $f \in \calF^k$.
		
Furthermore $\phi(d) \in \sum_{\mu \in \supp d}d(\mu)\varrho_*(\mu)$, 
therefore, $\phi(d)\neq 0$; and likewise $\phi(d')\neq0$.
		
It is clear that this implies the conclusion of Lemma~\ref{countable.hom.aux}.
\end{proof}
	
Now we are ready to prove Lemma \ref{countable.hom}.

\begin{proof}[Proof of Lemma \ref{countable.hom}]  
There are only a countably many of pairs $(T, n)\in \calT \times \omega$ such 
that $J_{T, n}\cap D\neq \emptyset$. 
We enumerate them faithfully as $((T_m, n_m): m\geq 2)$. 

For $m\geq 2$ let $\calF^m=\{h_\xi: \xi \in D\cap J_{T_m, n_m}\}$ 
and $p_m=p_{T_m, n_m}$. 
Let $p_0$ and $p_1$ be two ultrafilters that were not listed and 
let $\calF^0=\calF^1=\{r\}$. 
For each $m\geq 2$ and $\xi \in J_{T_m, n_m}\cap D$, let $\xi_{h_\xi, m}=\xi$. 
Let $h_{r, 0}=\chi_k$ and $h_{r, 1}=\chi_{k'}$ where $k, k' \in \omega$ are not 
in~$\supp d$. 
Then, by applying Lemma~\ref{countable.hom.aux} with $d'=\chi_k-\chi_{k'}$, 
there exist $\phi:\Z^{(D)}\to \T$ satisfying~(1) and~(2). 
To see it also satisfies~(3), notice 
that $\plim[p_0] \phi\circ r\neq \plim[p_1] \phi\circ r$.
\end{proof}

\section{Proof of Lemma \ref{step}}

In this section we present a proof of Lemma~\ref{step}. 
We will need the notion of integer stack, which was defined 
in~\cite{tomita2015}.

The integer stacks are collections of sequences in~$\Z^{(\cee)}$ that are 
usually associated to a selective ultrafilter. 
Given an finite set of sequences~$\calF$ it is possible to associate it to a 
integer stack which generates the same $\Q$~vector space as~$\calF$. 
The sequences in the stack have some nice properties that help us to construct
well behaved arcs when constructing homomorphisms, and the linear relations 
between $\calF$ and the sequences of the stack helps us to transform these 
arcs into arcs that work for the functions of $\calF$. 
Below, we give the definition of integer stack.

\begin{defin} \label{defin.stack} 
An \emph{integer stack $\calS$ on $A$} consists of
	
\begin{enumerate}[label=(\roman*)]
\item an infinite subset $A$ of $\omega$;
\item natural numbers $s$, $t$, and $M$; 
      positive integers $ r_i$ for $0\leq i<s$ and 
      positive integers $ r_{i,j}$ for $0\leq i<s$ and $0\leq j<r_i$;
\item functions $f_{i,j,k} \in (\Z^{(\cee)})^A$ for 
       $0\leq i< s$, $0\leq j<r_i$ and $0\leq k <r_{i,j}$ and
       elements $g_l\in(\Z^{(\cee)})^A$ for $0\leq l<t$;
\item sequences $\xi_i \in \cee^A$ for $0\leq i<s$ and 
      $\mu_l \in \cee^A$ for $0\leq l <t$ and
\item real numbers $\theta_{i,j,k} $ for $0\leq i<s$, $0\leq j<r_i$ 
      and $0\leq k <r_{i,j}$ 
\end{enumerate}

These are required to satisfy the following conditions:

\begin{enumerate}[label=(\arabic*)]
\item $\mu_l (n) \in \supp g_l(n)$ for each $n\in A$;
\item $\mu_{l^*}(n) \notin \supp g_{l} (n)$ for each $n \in A$ 
       and $0\leq l^*<l<t$;
\item the elements of $\{ \mu_l(n): 0\leq l<t$ and $n\in A\}$ are pairwise 
       distinct;
\item $|g_l(n)| \leq M$ for each $n \in A$ and $0\leq l <t$;
\item $\{\theta_{i,j,k}: 0 \leq k <r_{i,j}\}$ is a linearly independent subset 
      of~$\R$ as a $\Q$-vector space for each $0\leq i<s$ and $0\leq j<r_i$;
\item $\lim_{n\in A}\frac{f_{i,j,k}(n)(\xi_i(n))}{f_{i,j,0}(n)(\xi_i(n))}
       =\theta_{i,j,k}$ 
      for each $0\leq i<s$, $0\leq j<r_i$ and $0\leq k <r_{i,j}$;
\item the sequence $\bigl(|f_{i,j,k}(n)(\xi_i(n))|: n \in A\bigr)$
      diverges monotonically to~$\infty$,
      for each $0\leq i<s$, $0\leq j<r_i$ and $0\leq k <r_{i,j}$;
\item $|f_{i,j,k}(n)(\xi_i(n))|> |f_{i,j,k^*}(n)(\xi_i(n))|$ 
      for each $n\in A$, $i<s$, $j<r_i$ and $0\leq k<k^*<r_{i,j}$;
\item $\left(\frac{|f_{i,j,k}(n)(\xi_i(n))|}{|f_{i,j^*,k^*}(n)(\xi_i(n))|}:
       n \in A \right)$ converges monotonically to $0$ 
     for each $0\leq i<s$, $0\leq j^*< j<r_i$, $0\leq k<r_{i,j}$, and 
     $0\leq k^* <r_{i,j^*}$; and
\item $\{f_{i,j,k}(n)(\xi_{i^*}(n)) : n \in A \} \subseteq [-M,M]$ 
     for each $0\leq i^* <i <s$, $0\leq j<r_i$ and $0\leq k<r_{i,j}$.
\end{enumerate}
\end{defin}

It is not difficult to show that the sequences of the stack are linearly 
independent. 
Moreover, if $p$~is a free ultrafilter, $\calS$~is a stack over~$A$, 
and $A \in p$, then it is not difficult to see that 
$({[g_l]}_p: l<t)\cup ({[f_{i, j, k}]}_p: i<s, j<r_i, k<r_{i, j})$ is linearly 
independent in the $\Q$-vector space~$\Q^{(\cee)}/p$. 
We leave the details as an exercise to the reader.

\begin{defin} 
Given an integer stack $\calS$ and a natural number~$N$, 
the \emph{$Nth$ root of $\calS$}, written~$\frac{1}{N}\calS$, is obtained by
keeping all the structure in $\calS$ with the exception of the functions;
these are divided by~$N$. 
Thus a function~$f_{i,j,k} \in \calS$ is replaced 
by~$\frac{1}{N}f_{i,j,k}$ in~$\frac{1}{N} \calS$ 
for each $0\leq i <s$, $0\leq j<r_i$ and $0\leq k<r_{i,j}$ and 
a function~$g_l \in \calS$ is replaced by~$\frac{1}{N}   g_l$  
in  $\frac{1}{N}  \calS$ for each $0\leq l <t$.

A \emph{stack} (unqualified) is then defined to be the $Nth$ root of an 
integer stack for some positive integer~$N$.
\end{defin}

The lemma below gives the relation between a finite sequence of sequences 
in $\Z^{(\cee)}$ and a stack~$\calS$ that is associated to it. 
The first part of this lemma is proved in~\cite{tomita2015}. 
The second part was stated in~\cite{boero&castro&tomita2019} with no proof 
presented there, since it follows directly from statements of 
several lemmas and constructions from~\cite{tomita2015}. 
Since the construction there is long and complicated, we sketch in this paper,
for the sake of completeness, a proof for the second part by indicating 
which statements and proofs from~\cite{tomita2015} are used, without 
repeating the arguments.

\begin{lem} \label{stack} 
Let $h_0$, \dots, $h_{m-1}$ be sequences in $\Z^{(\cee)}$ 
and $\calU \in \omega^*$ a selective ultrafilter. 
Then there exists $A\in \calU$ and a stack~$\frac{1}{N} \calS$ on~$A$ 
such that:
if the elements of the stack have a $\calU$-limit in~$\Z^{(\cee)}$ then 
$h_i$~has a $\calU$-limit in~$\Z^{(\cee)}$ for each $0\leq i<m$.

We will say in this case that the finite sequence 
$\{ h_0, \ldots , h_{m-1}\}$ is associated to $(\frac{1}{N} \calS, A, \calU)$.
\begin{enumerate}
\item[$(\#)$] If $\{[h_0]_\calU, \ldots, [h_{m-1}]_\calU\}$ is a $\Q$-linearly 
              independent set and the group generated by it does not contain 
              nonzero constant classes, then each restriction~$h_i|_A$  is an 
              integer combination of the stack~$\frac{1}{N} \calS$ on~$A$. 
              On the other hand, each element of the integer stack~$\calS$ is 
              an integer combination of $\{h_0, \ldots, h_{m-1}\}$ restricted 
              to~$A$.
\end{enumerate}
\end{lem}

\begin{proof}
We prove $(\#)$. 
All numbered references in this proof are to the paper~\cite{tomita2015}.
 
First, notice that if $\{[h_0]_\calU, \ldots, [h_{m-1}]_\calU\}$ is 
a $\Q$-linearly independent set and the group generated by it does not 
contain nonzero constant classes, then it satisfies the conclusion of 
Lemma~4.1. 
Then, following the proof of Lemma~7.1, 
using the~$f$'s as the $h$'s themselves, we see that the 
functions $h_0$, \dots, $h_{m-1}$ are integer combinations of the 
stack~$\frac{1}{N}.\calS$ that was constructed. 

It remains to see that the functions of $\calS$ are integer combinations of 
the functions~$h_i$ restricted to~$A$. 
First, notice that in the statement of Lemma~5.4, 
by~x), xi), xii) and~xiv) the functions~$f_q^{i, j}$ and~$g_q^0$ are integer 
combinations of the~$h_i$. 
This Lemma is used in the proof of Lemma~5.5, where 
the functions~$f_q^{i, j}$ become the functions~$f_{i, j, k}$, so there are 
integer combinations of the $h_i$'s.

Now notice that in Lemma~6.1, by~g), c) and finite induction, the 
functions~$g_j^i$ are integer combinations of the~$h_i$, and some of these 
become the~$g_i$'s in the proof of Lemma~6.2.
As in the proof of~7.1 the stack is constructed by 
applying Lemma~5.5 or Lemma~6.2 
or Lemma~5.5 followed by Lemma~6.2 
(depending on the case), it follows that the stack constructed consists of 
functions that are linear combinations of functions the $h_i$'s 
(restricted to~$A$).
\end{proof}

Now we define some integers related to Kronecker's Theorem that will be 
useful in our proof. 
The existence of these integers follows from Lemma~4.3. of~\cite{tomita2015}. 
These integers were also defined and used in that paper.

\begin{defin} 
If $\{\theta_0, \ldots , \theta_{r-1}\}$ is a linearly independent subset of 
the $\Q$-vector space~$\R$ and $\epsilon>0$ then 
$L(\theta_0 , \ldots , \theta_{r-1},\epsilon)$ denotes a positive integer, $L$,
such that $\{(\theta_0 x +\Z, \ldots , \theta_{r-1} x +\Z): x \in I\}$ is 
$\epsilon$-dense in $\T^r$ in the usual Euclidean metric product topology, 
for any interval~$I$ of length at least~$L$.
\end{defin}

The last lemma we are going to need is Lemma~8.3 from \cite{tomita2015}, 
stated below.

\begin{lem} \label{homomorphism.step} 
Let $\epsilon$, $\gamma$ and $\rho$ be positive reals, 
$N$~a positive integer and $\psi$ be an arc function. 
Let $\calS$ be an integer stack on~$A \in [\omega]^{\omega}$ 
and $s$, $t$, $r_i$, $r_{i,j}$, $M$, $f_{i,j,k}$, $ g_l$, $\xi_i$, $\mu_j$ 
and~$\theta_{i,j,k}$ be as in Definition~\ref{defin.stack}. 

Let $L$ be an integer greater or equal to 
$\max\{L(\theta_{i,j,0}, \ldots ,\theta_{i,j,r_{i,j}-1}, \frac{\epsilon}{24}):
0\leq i<s$ and $0\leq j<r_i \}$ and 
let $r=\max \{r_{i,j}: 0\leq i<s$ and $0\leq j<r_i \}$.

Suppose that $n\in A$ is such that
\begin{enumerate}[label=(\alph*)]
\item $\{ V_{i,j,k}: 0\leq i <s, 0\leq j <r_i$ and $0\leq k <r_{i,j}\} 
      \cup \{W_l: 0\leq l<t\}$ is a family of open arcs of length~$\epsilon$;
\item $\delta (\psi(\beta))\geq \epsilon$ for each $\beta \in \supp \psi$;
\item $\epsilon>3N\cdot\rho\cdot
      \max\bigl(\{\|g_l(n)\|:0\leq l<t\} \cup
      \bigcup\{\|f_{i,j,k}(n)\|:0\leq i<s,0\leq j <r_i,0\leq k<r_{i,j}\}\bigr)$;
\item $3  M N s \gamma <\epsilon$;
\item $|f_{i,r_i-1,0}(n)(\xi_i(n))|\cdot \gamma > 3  L$ for each $0\leq i<s$;
\item $|f_{i,j-1,0}(n)(\xi_i(n))|\cdot
       \frac{\epsilon}{6\sqrt{r_{i,j}} |f_{i,j,0}(n)|} >3L$ for 
       each $0\leq i <s$ and $0<j<r_i$;
\item $\left|\theta_{i,j,k}-
        \frac{f_{i,j,k}(n)(\xi_i(n))}{f_{i,j,0}(n)(\xi_i(n))}\right|
        <\frac{\epsilon}{24\sqrt{r}L}$ 
        for each $i<s$, $j<r_i$ and $k<r_{i,j}$ and
\item $\supp \psi \cap \{ \mu_0(n), \ldots , \mu_{t-1}(n)\}=\emptyset$.
\end{enumerate}

Then there exists an arc function $\phi$ such that

\begin{enumerate}[label=(\Alph*)]
\item $N\cdot \phi(\beta)\subseteq 
       N\cdot \overline{\phi(\beta)}\subseteq 
       \psi(\beta)$ for each $\beta \in \supp \psi$;
\item $\sum_{\beta \in \supp g_l(n)} g_l(n)(\beta)  \phi(\beta) \subseteq W_l$ 
      for each $l <t$;
\item $\sum_{\beta \in \supp f_{i,j,k}(n)} f_{i,j,k}(n)(\beta) \cdot \phi(\beta) 
      \subseteq V_{i,j,k}$ for each $i<s$, $j<r_i$ and $k <r_{i,j}$;
\item $\delta(\phi(\beta))=\rho$ for each $\beta \in \supp \phi$ and
\item $\supp \phi$ can be chosen to be any finite set containing 
      \[
      \supp \psi \cup 
        \bigcup_{0\leq i<s, 0\leq j <r_i, 0\leq k<r_{i,j}}\supp f_{i,j,k}(n) \cup 
        \bigcup_{0\leq l <t} \supp g_l(n).
      \qed\]
\end{enumerate}
\end{lem}

Now we are ready to prove Lemma~\ref{step}.

\begin{proof}[Proof of Lemma~\ref{step}]
Write $\calF=\{u_0, \dots u_{q-1}\}$ with no repetition. 
Let $\calS$ be an integer stack on~$A'\in p$ and let $N$ be a positive integer 
such that $\left(\frac{1}{N}\calS , A', p \right)$ is associated to~$\calF$.

As in Definition~\ref{defin.stack} the components of~$\calS$ will be 
denoted $s$, $t$, $M$, $(r_i : i<s)$, 
$(r_{i,j}: i<s, j< r_i)$, 
$(f_{i,j,k}: i<s , j< r_i, k<r_{i,j})$, 
$(g_l: l<t)$, 
$( \xi_i: i<s)$, 
$(\mu^p: i<t)$ and 
$(\theta_{i,j,k}: 0\leq i < s, 0\leq j <r_i, k<r_{i,j})$.

We write $\{ f_{i,j,k}: i<s_p , j< r_i,  k<r_{i,j} \} \cup \{g_l: l<t\}$ 
as $\{v_0, \ldots , v_{q-1}\}$. 

Let $\calM$ be the $q \times q$ matrix of integer numbers such 
that $N u_i(n)= \sum_{j<q}\calM_{i,j}v_j(n)$ for all $n \in A$ and $i<q$. 

By $(\#)$ in Lemma~\ref{stack}, each $v_j$ is an integer combination of 
the~$u_i$'s, therefore the inverse matrix of $\frac{1}{N}\calM$, 
which we denote by~$\calN$, has integer entries.

\smallskip
Let $\epsilon' =\epsilon\cdot(\sum_{i,j<l}|\calM_{i,j}|)^{-1}$ and 
$\gamma <\epsilon'/(3  M N s)$. 
Let $L$ be larger than or equal to the maximum of the set 
$\{ L(\theta_{i,j,0}, \ldots , \theta_{i,j,r_{i,j}-1},\epsilon'/24): i<s, j<r_i\}$.

For each $n \in A'$, let $\delta_n<\frac{1}{2}$ be such that:
$$
\epsilon' > 3 N  \cdot
\max\bigl(\{\|g_l(n)\|: 0\leq l <t\} \cup 
    \bigcup \{\|f_{i,j,k}(n)\|: 0\leq i<s, 0\leq j <r_i, 0\leq k<r_{i,j}\}\bigr)
 \cdot\frac{\delta_n}{N}
$$
We note that both $N$'s above cancel but we write this way as we will 
use~${\delta_n}/{N}$ in the place of~$\rho$ in item~c) of 
Lemma~\ref{homomorphism.step}.

Let $r=\max \{r_{i,j}: 0\leq i<s,  0\leq j<r_i \}$. 
Let $A$ be the set of $n$'s in $A'$ such that:

\begin{itemize}
\setlength\itemsep{1.6em}
    \item $\displaystyle| f_{i,r_i-1,0}(n)(\xi_i(n))|  \gamma > 3  L\text{ for each } 0\leq i<s,$,
    
    \item $\displaystyle|f_{i,j-1,0}(n)(\xi_i(n))|\cdot
       \frac{\epsilon'}{6\sqrt{r_{i,j}} |f_{i,j,0}(n)|} >3L$ for 
       each $0\leq i <s$ and $0<j<r_i$,

\item $\displaystyle\left|\theta_{i,j,k}-\frac{f_{i,j,k}(n)(\xi_i(n))}{f_{i,j,0}(n)(\xi_i(n))}\right|<
\frac{\epsilon'}{24\sqrt{r}L} 
\text{ for each }i<s, j<r_i \text{ and }k<r_{i,j},$ and

\item $\displaystyle E \cap \{ \mu_0(n), \ldots , \mu_{t-1}(n)\}=\emptyset.
$
\end{itemize}

Notice that $A$ is cofinite in $A'$, therefore $A \in p$.

We claim this $A$ and this sequence $(\delta_n: n \in A)$ work. 

Fix $n \in A$.

Let $(U_f: f \in \calF)$ be a family of arcs of length~$\epsilon$ and 
let $\varrho$ be an arc function of length at least~$\epsilon$ with 
$\supp \varrho\subseteq E$. 
We rewrite the family of arcs as $(U_i: i<q)$, where $U_i=U_{f_i}$ for 
each $i<q$. 
For each $i<q$ let $y_i$ be a real such that $y_i+\Z$ is the center of~$U_i$. 
Let $z_j=\sum_{i<q}\calN_{j, i} \frac{y_i}{N}$ and, for each~$j$ let $R_j$ be 
the arc of center~$z_j$ and length~$\epsilon'$. 
Since $\calN$~is a matrix of integers, 
$z_j+\Z=\sum_{i<q}\calN_{j, i} (\frac{y_i}{N}+\Z)$.  
Then the arc $\sum_{j<q}\calM_{i,j}R_j$ is a subset of~$U_i$ for each $i<q$.

Now we aim to apply Lemma~\ref{homomorphism.step}. 
Set $\psi=\varrho$, $\rho=\delta_n/N$ and $\epsilon'$ in the place 
of~$\epsilon$. 
For $i<s$, $j<r_i$, $k<r_{i, j}$ we put $V_{i, j, k}=R_x$ if $f_{i, j, k}=v_x$ 
for some $x<q$, 
and for $j<t$ we put $W_j=R_x$ if $g_{j}=v_x$ for some $x<q$.

Then there exists an arc function $\tilde{\psi_n}$ such that

\begin{enumerate}[label=(\Alph*)]
\item $N\tilde{\psi_n}\subseteq N \overline{\tilde{\psi_n}}
       \subseteq \varrho(\beta)$ for each $\beta \in \supp \psi$;
\item $\sum_{\beta \in \supp g_l(n)} g_l(n)(\beta) \tilde{\psi_n}(\beta) 
       \subseteq W_l$ for each $l <t$;
\item $\sum_{\beta\in\supp f_{i,j,k}(n)}f_{i,j,k}(n)(\beta)\cdot\tilde{\psi_n}(\beta) 
       \subseteq V_{i,j,k}$ for each $i<s$, $j<r_i$ and $k <r_{i,j}$;
\item $\delta(\tilde{\psi_n}(\beta))={\delta_n}/N$ 
       for each $\beta \in \supp \tilde{\psi_n}$ and
\item $\supp \tilde{\psi_n}$ is equal to
      \[
      \bigcup_{0\leq i<s, 0\leq j <r_i, 0\leq k<r_{i,j}}\supp f_{i,j,k}(n) \cup 
      \bigcup_{0\leq l <t} \supp g_l(n)\cup E=
      \bigcup_{f \in \calF}\supp f(n) \cup E.
\]
\end{enumerate}

Let $\psi_n=N\tilde{\psi_n}$. 
By~(A), $\psi_n\leq \varrho$. 
By~(E) and~(D), $\supp \psi_n=\bigcup_{f \in \calF} \supp f(n)\cup E$ and for 
each $\beta \in \supp \psi_n$, we have $\delta(\psi_n(\beta))=\delta_n$. 
Let $S=\supp \psi_n$. 
Now notice that given $u_i \in \calF$ we have:
\begin{align*}
\sum_{\mu \in \supp u_i }u_i(n)(\mu)\psi_n(\mu)
&=\sum_{\mu \in S }u_i(n)(\mu)N\tilde\psi_n(\mu)\\
&=\sum_{\mu \in S}\left( \sum_{j<q}\calM_{i,j}v_j(n)(\mu)\right)\tilde\psi_n(\mu)\\
&=\sum_{j<q}\calM_{i,j}\left( \sum_{\mu \in S}v_j(n)(\mu)\tilde\psi_n(\mu)\right)
\end{align*}
Then by (B), (C) and the definitions of the $W_l$'s and $V_{i, j, k}$'s:
$$
\sum_{\mu \in \supp u_i }u_i(n)(\mu)\psi_n(\mu)=\sum_{\mu \in S }u_i(n)(\mu)N\tilde\psi_n(\mu)\subseteq\sum_{j<q}\calM_{i,j}R_j\subseteq U_i.
$$ 
As intended.
\end{proof}

\section{Final comments}
	
The method to construct countably compact free Abelian groups came from the technique to construct countably compact groups without non-trivial convergent sequences. It is not known if there is an easier method to produce countably compact group topologies on free Abelian groups if we do not care if the resulting topology has convergent sequences. 

In fact, even to produce a countably compact group topology with convergent 
sequences in non-torsion groups it is used a modification of the technique 
to construct 
countably compact groups without non-trivial convergent sequences, see 
\cite{bellini&boero&castro&rodrigues&tomita} and
\cite{boero&castro-pereira&tomitaoneselective}.
	
The first examples of countably compact groups without non-trivial convergent
sequences were obtained by Hajnal and Juh\' asz \cite{hajnal&juhasz} 
under~CH.  
E. van~Douwen \cite{vandouwen2} obtained an example from~MA and asked for 
a ZFC~example. 
Other examples  were obtained using $\mathrm{MA_{countable}}$ 
\cite{koszmider&tomita&watson}, 
a selective ultrafilter \cite{garcia-ferreira&tomita&watson}
and in the Random real model \cite{szeptycki&tomita}. 
Only recently, Hru\v sak, van Mill, Shelah and Ramos obtained an example 
in ZFC (\cite{hrusak&vanmill&ramos&shelah}).

This motivates the following questions in ZFC:
 
\begin{question}  
Are there large countably compact groups without non-trivial convergent 
sequences in~ZFC? 
Is there an example of cardinality~$2^\cee$?
\end{question}
 
The example of Hru\v sak et~al has size continuum and it is not clear if 
their construction could yield larger examples.
 
\begin{question}
Is there a countably compact free Abelian group in~ZFC? 
A countably compact free Abelian group without non-trivial convergent 
sequences in~ZFC?
\end{question}
  
It is still open if there exists a torsion-free group in ZFC that admits a 
countably compact group topology without non-trivial convergent sequences. 
If such example exists then there is a countably compact group topology 
without non-trivial convergent sequences in the free Abelian group of 
cardinality~$\cee$ 
(see \cite{tomitaTopApplsquare2005} or \cite{tomitaTopAppl2019ZFC}).
  
\begin{question}
Is there a both-sided cancellative semigroup that is not a group that admits 
a countably compact semigroup topology (a Wallace semigroup) in~ZFC?
\end{question}
 
The known examples were obtained 
in~\cite{robbie&svetlichny} under~CH,  
in~\cite{tomita} under~$\mathrm{MA_{countable}}$, 
in~\cite{madariaga-garcia&tomita} from $\cee$~incomparable selective 
ultrafilters and 
in~\cite{boero&castro-pereira&tomitaoneselective} from one selective ultafilter.
The last two use the known fact that a free Abelian group without non-trivial 
convergent sequences contains a Wallace semigroup, which was used 
in~\cite{robbie&svetlichny}. 
The example in~\cite{tomita} was a modification of~\cite{hart&vanmill}.

\bibliographystyle{amsplain}
\bibliography{gruposenumeravelmentecompactos}
	
\end{document}